\documentclass[12pt]{amsart}
\usepackage{amssymb}
\usepackage{rotating}
\usepackage{cite}
\usepackage{hyperref}
\newcommand{\ad }{\mathrm{ad}\,}
\newcommand{\fie }{\Bbbk}
\newcommand{\dpu }{\hat{u}}   % divided power u

\newcommand{\lex }{<_{\mathrm{lex}}}

\newcommand{\NA }{\mathcal{B}}
\newcommand{\ndN }{\mathbb{N}}
\newcommand{\ndZ }{\mathbb{Z}}
\newcommand{\ndJ }{\mathbb{J}}
\newcommand{\ot }{\otimes}

\newcommand{\su }{S}
\newcommand{\tabincell}[2]{\begin{tabular}{@{}#1@{}}#2\end{tabular}}
	
\newtheorem{theo}{Theorem}[section]
\newtheorem{prop}[theo]{Proposition}
\newtheorem{lemm}[theo]{Lemma}
\newtheorem{coro}[theo]{Corollary}
\theoremstyle{definition}
\newtheorem{defi}[theo]{Definition}
\newtheorem{exam}[theo]{Example}
\newtheorem{rema}[theo]{Remark}

\begin{document}
\title[Root multiplicities for Nichols algebras]
{Root multiplicities for Nichols algebras of diagonal type of rank two}

\author{I. Heckenberger}
\address{Philipps-Universit\"at Marburg,
FB Mathematik und Informatik,
Hans-Meerwein-Stra\ss e,
35032 Marburg, Germany.}
\email{heckenberger@mathematik.uni-marburg.de}

\author{Y. Zheng}
\address{Department of mathematics East China Normal University,Shanghai 200241,China.}
\email{52150601007@ecnu.cn}
\thanks{The second named author was supported by China Scholarship Council}

\date{}

\maketitle

\begin{abstract}
  We determine the multiplicities of a class of roots for
	Nichols algebras of diagonal type of rank two, and identify the corresponding
	root vectors. Our analysis is based on a precise description of the relations
	of the Nichols algebra in the corresponding degrees.

\end{abstract}

\section{Introduction}

Since the introduction of Nichols algebras in the late 1990-ies, the topic
developed to an own-standing research field with many relationships to different
(mainly algebraic or combinatorial) fields in mathematics. In particular,
Nichols algebras
are heavily used for the study of pointed Hopf algebras. Although Nichols
algebras can be defined in any suitable braided monoidal category, a big part of the
theory is dominated by Nichols algebras of diagonal type.

By now, a deep
understanding of the structure of finite-dimensional Nichols algebras of
diagonal type is available, based on the existence of a PBW basis \cite{khar} and the notion
of roots \cite{H2006}.
In the general setting, one is constantly tempted to seek for relationships
with Kac-Moody and Borcherds Lie (super) algebras.
The latter seems to be very strong in the finite case because of the
definitions of real roots in the two theories. However, the knowledge about
imaginary roots and their multiplicities is little in the case of Kac-Moody
algebras, and even poorer for Nichols algebras of diagonal type.
For information on recent activities in the theory of Kac-Moody algebras we refer
to \cite{AKV2016}.
With our
results
we make a small step towards a better understanding of the Nichols algebra theory in this
respect.

In this paper, we concentrate on Nichols algebras of
diagonal type of rank two. In order to clarify the context, we introduce the
notion of root vector candidates and root vectors.
We focus on the special roots $m\alpha_1+2\alpha_2$, where $m\in \ndN_0$ and
$\alpha_1, \alpha_2$ is the standard basis of $\ndZ^2$.
(The root multiplicities of $m\alpha_1+k\alpha_2$ with $m\in \ndN_0$, $k\in
\{0,1\}$, have been known before.)
We identify the family $(P_k)_{k\in \ndN _0}$ in the free algebra over a
two-dimensional braided vector space $V$ of diagonal type, and relate the relations
in the Nichols algebra of $V$ of degree $m\alpha_1+2\alpha_2$ to this family.
We find two of our results particularly interesting.
First, in Proposition~\ref{cor:48} we prove that if a root vector
candidate is a root vector, then any lexicographically larger root vector candidate
of the same degree is a root vector, too.
Second, in Theorem~\ref{main} we describe precisely when
a root vector candidate is a root vector. To do so, we define a subset $\ndJ$
of $\ndN _0$ depending on the given braiding, which measures the multiplicities
of all roots of the form $m\alpha_1+2\alpha_2$ in a simple way. For the
calculation of $\ndJ$ one needs only elementary (and simple) calculations with
Laurent polynomials in three indeterminates. Unfortunately, the proof of this
theorem requires that we work over a field of characteristic $0$.

The paper is organized as follows.
In Section~\ref{se:equations}, we give some equations for Gaussian binomial
coefficients, which will be needed later.
In Section~\ref{se:prelims}, we recall some
fundamental definitions and results on which our work is based.
In Section~\ref{se:multi} we formulate and prove our main results mentioned
above. We also conclude a non-trivial lower bound on root multiplicities.

The authors thank Eric Heymann-Heidelberger for interesting discussions on the
topic. The paper was written during the visit of the second author to Marburg University
supported by China Scholarship Council. And the second author thanks department of
  FB Mathematik and Informatik of  Marburg University for hospitality.

\section{Quantum integers and Gaussian binomial coefficients}
\label{se:equations}

Throughout the paper let $\ndN $ denote the set of positive integers and let
$\ndN _0=\ndN \cup\{0\}$. We write $\ndZ $ for the set of integers.

For our study of Nichols algebras we will need some non-standard formulas
for quantum integers and Gaussian binomial coefficients.

In the ring $\ndZ[q]$, let $(0)_q=0$ and for any $m\in \ndN$, let
$$(m)_q=1+q+q^2+\cdots+q^{m-1}$$
and $(-m)_q=-(m)_q$.
The polynomials $(m)_q$ with $m\in \ndZ$ are also known as quantum integers.
Moreover, let $(0)_q^!=1$, and for any $m\in \ndN$ let
$(m)_q^!=\prod_{i=1}^m(i)_q$.
For any $i,m\in \ndZ $ with $0\le i\le m$, the rational function
$${m\choose i}_q=\frac{(m)_q^!}{(i)_q^!(m-i)_q^!} \in \mathbb{Q}(q) $$
is in fact an element of $\ndZ [q]$ and is called a Gaussian binomial
coefficient. For $m\in \ndN_0$, $i\in \ndZ$ with $i<0$ or $i>m$ one defines
${m\choose i}_q=0$.
The Gaussian binomial coefficients satisfy the following formulas:
\begin{align}
	\label{eq:qbinomial1}
	{m\choose i}_q=&\,{m\choose m-i}_p,\\
	\label{eq:qbinomial2}
	{m\choose i}_q=&\,q^i{m-1\choose i}_q+{m-1\choose i-1}_q,\\
	\label{eq:qbinomial3}
	{m\choose i}_q=&\,{m-1\choose i}_q+q^{m-i}{m-1\choose i-1}_q
\end{align}
for any $m\in \ndN $, $i\in \ndZ $.

\begin{lemm} \label{le:binomsum}
	Let $t\in \ndN_0$ and $k\in \ndZ $ with $k\ge -1$. Then
	$$
	\sum_{j=t}^kq^{-j(j+1)/2}q^{(j-t)(j-t-1)/2}{j\choose t}_q
	=q^{-(t+1)(2k-t)/2}{k+1\choose t+1}_q.
	$$
\end{lemm}

\begin{proof}
	We proceed by induction on $k$. For $k<t$ the claim is trivial.
	Assume now that $t,k\in \ndN_0$ and that the claim holds for $t$ and $k-1$.
	The summand for $j=k$ on the left hand side is $q^{-(t+1)(2k-t)/2}{k\choose
	t}_q$. By subtracting this from both sides of the equation and using
	Equation~\eqref{eq:qbinomial2}, the claim follows from the induction
	hypothesis.
\end{proof}

\begin{lemm} \label{le:lesscrazysum}
	Let $m\in \ndN_0$ and $n\in \ndZ$. Then in $\ndZ[q,t]$ we have
 $$\sum_{i=0}^m{m\choose i}_{\!\!q}q^{i(i-1)/2}
 \prod_{j=0}^{i-1}(q^{j+n}t^2-t)\prod_{j=1}^{m-i}(1-q^{m+n-j}t^2)
 =\prod_{j=0}^{m-1}(1-q^jt).$$
\end{lemm}

 \begin{proof}
We prove the claim by induction on $m$.

For $m=0$ the claim is trivial.
Now assume that the claim holds for some $m\in \ndN_0$. By applying
Equation~\eqref{eq:qbinomial3} we obtain that
\begin{align*}
 & \sum_{i=0}^{m+1}{{m+1}\choose i}_{\!\!q}q^{i(i-1)/2}
  \prod_{j=0}^{i-1}(q^{j+n}t^2-t)\prod_{j=1}^{m+1-i}(1-q^{m+n+1-j}t^2)\\
 &=\sum_{i=0}^{m}{{m}\choose i}_{\!\!q}q^{i(i-1)/2}
  \prod_{j=0}^{i-1}(q^{j+n}t^2-t)\prod_{j=1}^{m+1-i}(1-q^{m+n+1-j}t^2)\\
 &+\sum_{i=1}^{m+1}{{m}\choose {i-1}}_{\!\!q}q^{m+1-i}q^{i(i-1)/2}
 \prod_{j=0}^{i-1}(q^{j+n}t^2-t)
 \prod_{j=1}^{m+1-i}(1-q^{m+n+1-j}t^2).
\end{align*}
Regarding the first term, note that
$$ \prod_{j=1}^{m+1-i}(1-q^{m+n+1-j}t^2)
=(1-q^{m+n}t^2)\prod_{j=1}^{m-i}(1-q^{m+n-j}t^2).
$$
Hence, by induction hypothesis, the first term is equal to
$$(1-q^{m+n}t^2)\prod_{j=0}^{m-1}(1-q^jt).$$
Moreover, the second term is equal to
\begin{align*}
 \sum_{i=0}^{m}{{m}\choose {i}}_{\!\!q}q^{m-i}q^{i(i+1)/2}
  \prod_{j=0}^{i}(q^{j+n}t^2-t)\prod_{j=1}^{m-i}(1-q^{m+n+1-j}t^2).
\end{align*}
Now by using that
$$ q^{m-i}q^{i(i+1)/2}\prod_{j=0}^i(q^{j+n}t^2-t)
=q^m q^{i(i-1)/2}(q^nt^2-t)\prod_{j=0}^{i-1}(q^{j+n+1}t^2-t),
$$
induction hypothesis implies that the second term is equal to
\begin{align*}
	q^m(q^nt^2-t)\prod_{j=0}^{m-1}(1-q^jt).
\end{align*}
Since $1-q^{m+n}t^2+q^m(q^nt^2-t)=1-q^mt$, the claim holds for $m+1$.
\end{proof}

\begin{lemm}\label{le:Q1Q2}
Let $k, m\in \ndN _0$ and
\begin{align*}
	&Q_1^{k,m}=\sum_{i=0}^{m}{{m+1}\choose
 {i}}_qq^{i(2k+i-1)/2}\prod_{j=0}^{i-1}(q^{k+j}r^2-r)
 \prod_{j=1}^{m-i}(1-q^{2k+m-j}r^2)\\
 &Q_2^{k,m}=
 \frac{q^{(2k+m)(m+1)/2}(-r)^{m+1}-1}{q^{2k+m}r^2-1}
 \prod_{i=0}^{m}(1-q^{k+i}r)
\end{align*}
in $\ndZ[q,r]$. Then $Q_1^{k,m}=Q_2^{k,m}$.
\end{lemm}

\begin{proof}
	Clearly, $Q_1^{k,m}\in \ndZ[q,r]$.
	If $m$ is odd then $q^{2k+m}r^2-1$ divides $(q^{2k+m}r^2)^{(m+1)/2}-1$ in
	$\ndZ[q,r]$. If $m$ is even, then
	$$q^{2k+m}r^2-1=(q^{k+m/2}r-1)(q^{k+m/2}r+1).$$
	Since $q^{k+m/2}r-1$ divides $\prod_{i=0}^m(1-q^{k+i}r)$ and
	$q^{k+m/2}r+1$ divides $-(q^{k+m/2}r)^{m+1}-1$ in $\ndZ[q,r]$, we conclude
	that $Q_2^{k,m}\in \ndZ[q,r]$. Moreover,
\begin{align*}
	&Q_1^{k,m}(1-q^{2k+m}r^2)\\
&=\sum_{i=0}^{m}{{m+1}\choose {i}}_{\!\!q}q^{i(2k+i-1)/2}
  \prod_{j=0}^{i-1}(q^{k+j}r^2-r)
  \prod_{j=0}^{m-i}(1-q^{2k+m-j}r^2)\\
&=\sum_{i=0}^{m+1}{{m+1}\choose {i}}_{\!\!q}q^{i(2k+i-1)/2}
  \prod_{j=0}^{i-1}(q^{k+j}r^2-r)
  \prod_{j=1}^{m+1-i}(1-q^{2k+m+1-j}r^2)\\
  &\quad -q^{(m+1)(2k+m)/2}\prod_{j=0}^{m}(q^{k+j}r^2-r).
\end{align*}
By Lemma~\ref{le:lesscrazysum} for $m+1$, $n=0$ and $t=q^kr$, the first
term is equal to $\prod_{j=0}^m(1-q^{j+k}r)$.
{}From this it follows that $Q_1^{k,m}=Q_2^{k,m}$.
\end{proof}

\section{Preliminaries on Nichols algebras}
\label{se:prelims}

In the remaining part of the paper, let $\fie$ be a field
and let $\fie^\times=\fie\setminus \{0\}$.

%\subsection{Root vectors and multiplicities.}
We start with collecting some information on Lyndon words, which is fairly
standard.

Let $A$ be a finite set (called the alphabet) and let $\mathbb A$ and
$\mathbb A^\times$ denote the set of words and nonempty words, respectively,
with letters in $A$. For any $s\in \ndN_0$, $a_1, a_2,\dots,a_s \in A$ and
$u=a_1\cdots a_s \in \mathbb A$ we write $|u|=s$ and call $s$ the
\textbf{length} of $u$.

We fix a total ordering $<$ on $A$. It induces a total ordering $\lex$ on $\mathbb A$
called the lexicographic ordering:
Two elements $u,v\in \mathbb A$ satisfy $u\lex v$ if and only if either
$v=uw$ for some $w\in \mathbb A^\times $,
or there exist $w,u',v'\in \mathbb A$ and $a,b \in A$
such that $u=wau'$, $v=wbv'$, and $a<b$.

We say that a word $u\in \mathbb A^\times$ is a \textbf{Lyndon word} if for any
decomposition $u=wv$, $w,v\in \mathbb A^\times$, the relation $u\lex vw$ holds.

A word $u\in \mathbb A^\times$ is a Lyndon word if and only if either $u\in
A$, or there exist Lyndon words $w,v\in \mathbb A^\times $ such that $w\lex v$ and $u=wv$.

Any Lyndon word $u$ of length at least two
has a unique decomposition into the product of two Lyndon words
$u=wv$, where $|w|$ is minimal. It is called the \textbf{Shirshow decomposition}
of $u$.

The theory of Lyndon words is used in \cite{khar} to define PBW bases of Nichols algebras of
diagonal type. (In fact, in \cite{khar} a much more general situation is
considered.)

Let $n\in \ndN $ and let $(V,c)$ be an $n$-dimensional braided vector space of
diagonal type. Let $I=\{1,\dots,n\}$, and let
$(q_{ij})_{i,j\in I}\in (\fie^\times)^{n\times n}$
and
$x_1,\dots,x_n$ be a basis of $V$ such that
$$ c(x_i\otimes x_j)=q_{ij}x_j\otimes x_i$$
for any $i,j\in I$.
Let $T(V)$ and $\NA(V)$ denote the tensor algebra and the Nichols algebra of
$V$, respectively.
For the basics of the theory of Nichols algebras we refer to
\cite{inp-AndrSchn02}.
We write $\pi:T(V)\to \NA (V)$ for the canonical map.

Let $\alpha_1,\dots ,\alpha_n$ be the standard basis of $\ndZ^n$
and $\chi:\ndZ^n \times \ndZ^n \to \fie^\times $ be the bicharacter
on $\ndZ^n$ such that $\chi(\alpha_i,\alpha_j)=q_{ij}$ for any $i,j\in I$.

Let $X=\{x_1,\dots,x_n\}$, and fix the total ordering on $X$ such that
$x_i<x_j$ whenever $1\le i<j\le n$.
Let $\mathbb X$ and $\mathbb X^\times$ denote the set of words and non-empty
words over the alphabet $X$, respectively.
The elements of $\mathbb X$ can naturally be viewed as elements of (any
quotient of) $T(V)$,
and as such they form a vector space basis of $T(V)$.
Both $T(V)$ and $\NA (V)$ have a unique $\ndZ^n$-graded braided bialgebra
structure such that $\deg (x_i)=\alpha_i$ for any $i\in I$.
In particular, for any $k\in \ndN _0$ and $l_1,\dots,l_k \in I$
the degree of $x_{l_1}\cdots x_{l_k}$ is
$\sum_{i=1}^k\alpha_{l_i}$.
We write $\deg(x)$ for the degree of any homogeneous element $x$ of $T(V)$ or
$\NA (V)$.

For a Lyndon word $u\in \mathbb X^\times$, following \cite{khar} we define
the \textbf{super-letter} $[u]\in\NA(V)$ inductively as follows:
\begin{enumerate}
 \item $[u]=u$, if $u\in X$, and
\item $[u]=[v][w]-\chi(\deg(v),\deg(w))[w][v]$
  if $u\in \mathbb{X}^\times$, $|u|\ge 2$, and  $u=vw$ is the Shirshow
  decomposition of $u$.
\end{enumerate}
Moreover, for any Lyndon word $u$ and any integer $k\ge 2$ let $[u^k]=[u]^k$.

The total ordering on $\mathbb{X}$ induces a total ordering on the set of super-letters:
$$[u]<[v]\,\Leftrightarrow\, u\lex v.$$

For any $\alpha \in \ndZ^n$, let $o_\alpha \in \ndN \cup \{\infty\}$
be the multiplicative order of $\chi(\alpha ,\alpha )\in \fie^\times $.
Moreover, let
\begin{align*}
	O_\alpha =\begin{cases}
		\{1,o_\alpha,\infty \} & \text{if $o_\alpha=\infty $ or
			$\mathrm{char}(\fie)=0$,}\\
		\{1,o_\alpha p^k,\infty \mid k\in \ndN_0 \} &
		\text{if $o_\alpha <\infty $, $p=\mathrm{char}(\fie )>0$.}
	\end{cases}
\end{align*}
Kharchenko proved the following fundamental result on Nichols algebras.

\begin{theo} \cite{khar}
  There exists a set $L$ of Lyndon words and a function $h:L\to \ndN\cup
	\{\infty\}$, where $h(v)\in O_{\deg v}\setminus \{1\}$ for any $v\in L$,
  such that the elements
  \begin{align*}
    [v_k]^{m_k}\cdots [v_1]^{m_1},\quad &
    \text{$k\in \ndN _0$, $v_1,\dots,v_k\in L$, $v_1\lex v_2\lex \cdots \lex v_k$,}\\
    & \text{$0< m_i<h(v_i)$ for any $i$,}
  \end{align*}
  form a vector space basis of $\NA (V)$.
\end{theo}

In fact, the set $L$ and the function $h$ in the above theorem are uniquely
determined.

In some situations it is more appropriate to work with a slightly different
presentation of the above basis of $\NA (V)$, in which the function $h$
does not appear.

\begin{defi}
Let $w\in\mathbb X^\times $. We say that $[w]$ is a
\textbf{root vector candidate} if
$w=v^k$ for some Lyndon word $v$
and $k\in O_{\deg v}\setminus \{\infty \}$.
\end{defi}

\begin{defi} \label{de:rootvector}
A root vector candidate $[w]$, where $w\in \mathbb X^\times $,
is called a \textbf{root vector (of $\NA (V)$)}
if $[w]\in \NA (V)$ is not a
linear combination of elements of the form
$[v_k]^{m_k}\cdots [v_1]^{m_1}$, where $k\in \ndN_0$ and
$[v_1],\dots,[v_k]$ are root vector candidates with
$w\lex v_1\lex \cdots \lex v_k$.
\end{defi}

\begin{rema} \label{re:rvcrit}
  By \cite[Corollary~2]{khar}, for any Lyndon word $w\in \mathbb X^\times$
the root vector candidate $[w]$ is a root vector if and only if
$w\in \NA (V)$ is not a
linear combination of elements of the form
$[v_k]^{m_k}\cdots [v_1]^{m_1}$, where $k\in \ndN_0$ and
$[v_1],\dots,[v_k]$ are root vector candidates with
$w\lex v_1\lex \cdots \lex v_k$.
\end{rema}

Note that in Definition~\ref{de:rootvector} it is not necessary to put assumptions
on the degrees of the monomials, since $\NA (V)$ is graded.

\begin{exam} \label{ex:Lyndon2}
	Assume that $n\ge 2$.
	Let $k\in \ndN_0$.
	The only Lyndon word of degree $k\alpha_1+\alpha_2$ in $\mathbb X$ is
	$x_1^kx_2$,
	and the only root vector candidate of degree $k\alpha_1+\alpha_2$ is
	$[x_1^kx_2]$. Since $\NA (V)$ is $\ndN_0^n$-graded,
	$[x_1^kx_2]$ is not a root vector if and only if $[x_1^kx_2]=0$ in $\NA (V)$.
        In our setting, the latter can be characterized in terms of the matrix
        $(q_{ij})_{i,j\in I}$ using Rossos Lemma \cite[Lemma~14]{a-Rosso98}:
  For any $k\ge 0$,
   \begin{align*}
    [x_1^{k+1}x_2]=0\quad \Leftrightarrow \quad & (k+1)_{q_{11}}^!
    \prod_{i=0}^k(1-q_{11}^iq_{12}q_{21})=0.
  \end{align*}

	The Lyndon words of degree $k\alpha_1+2\alpha_2$ in $\mathbb X$ are
	the words	$x_1^{k_1}x_2x_1^{k_2}x_2$ with $k_1,k_2\in \ndN_0$, $k_1+k_2=k$,
	$k_1>k_2$. The elements $[x_1^{k_1}x_2x_1^{k_2}x_2]$
	are the only root vector candidates of degree $k\alpha_1+2\alpha_2$, except
	when $k$ is even and $q_{11}^{k^2/4}(q_{12}q_{21})^{k/2}q_{22}=-1$.
	In the latter
	case, $[x_1^{k/2}x_2]^2$ is the only additional root vector candidate of
	degree $k\alpha_1+2\alpha_2$. The definition implies that the element
	$[x_1^{k_1}x_2x_1^{k_2}x_2]$ with $k_1+k_2=k$, $k_1\ge k_2$,
	is not a root vector if and only if there exists a relation in $\NA (V)$
	of the form
	$$ \sum_{i=k_2}^{k_1}\lambda_i [x_1^ix_2][x_1^{k-i}x_2]=0 $$
	such that $\lambda _i\in \fie $ for any $k_2\le i\le k_1$ and
	$$\lambda_{k_1}=1,\quad
	  \lambda_{k_2}=-\chi(k_1\alpha_1+\alpha_2,k_2\alpha_1+\alpha_2).
	$$
	(This is also true if $k_1=k_2$!)
\end{exam}

Note that the definitions of a root vector candidate and a root vector depend
on the bicharacter $\chi $.
Now Kharchenko's theorem can be restated as follows.

\begin{theo}
	Let $L\subseteq \mathbb X^\times$ such that $w\in L$ if and only if $[w]$ is a
	root vector.
  Then the elements
  \begin{align*}
    [v_k]^{m_k}\cdots [v_1]^{m_1},\quad &
    \text{$k\in \ndN _0$, $v_1,\dots,v_k\in L$, $v_1\lex v_2\lex \cdots \lex v_k$,}\\
		& \text{$0< m_i<\min (O_{\deg v_i}\setminus \{1\})$ for any $i$,}
  \end{align*}
  form a vector space basis of $\NA (V)$.
\end{theo}

(Note that $\min (O_{\alpha }\setminus \{1\})$ for $\alpha \in \ndZ^n$ equals
$o_\alpha$, except when $\alpha=1$.)
This reformulation of Kharchenko's theorem allows to define the set
$$\boldsymbol{\Delta}_+ = \{\deg (u)\mid u\in L\}$$
of \textbf{positive roots of} $\NA (V)$ and the \textbf{root system}
$\boldsymbol{\Delta}=\boldsymbol{\Delta}_+ \cup -\boldsymbol{\Delta}_+$ of
$\NA(V)$, see \cite{H2006}.
It turns out that this definition is independent of choices.
For any $\alpha \in \boldsymbol{\Delta}_+$, the number of elements $u\in L$
with $\deg(u)=\alpha$ is called the \textbf{multiplicity of} $\alpha $.

One of the biggest open problems in the theory of Nichols algebras of diagonal
type is to determine for any $V$ (in a suitable class) the set $L$ in
Kharchenko's theorem. In this paper we determine for $n=2$ the subset of $L$ of
elements of degree $m\alpha_1+2\alpha_2$, where $m\in \ndN $.

Next we recall standard tools for working with Nichols algebras of diagonal type.

For any $i \in \{1,\ldots,n\}$, there exists a unique skew-derivation $d_i$
of the tensor algebra $T(V)$ such that
\begin{align} \label{eq:dix}
 d_i(x_j)=\delta_{ij},\quad
 d_i(xy)=d_i(x)y+\chi(\alpha,\alpha_i)xd_i(y)
\end{align}
for any $j\in I$ and $x,y\in T(V)$ with $\deg (x)=\alpha $.
These skew-derivations induce skew-derivations of $\NA (V)$ which will be
denoted by the same symbols.

\begin{rema} \label{re:kerd}
	An element $x\in \NA (V)$ is constant if and only if $d_i(x)=0$ in $\NA (V)$
	for any $i\in I$. In particular, a homogeneous element $x\in \NA (V)$ of
	non-zero degree is zero if and only if $d_i(x)=0$ in $\NA (V)$ for any $i\in
	I$. Because of this, the skew-derivations $d_i$, $i\in I$, and their relatives
	belong to the main tools in the study of
	Nichols algebras of diagonal type.
\end{rema}

Since $\NA (V)$ is an $\ndN_0^n$-graded coalgebra,
for any $x\in \NA (V)$ and any $\beta,\gamma\in \ndN_0^n$
there exist uniquely determined elements
$$x_{\beta,\gamma }\in \NA(V)(\beta)\otimes \NA (V)(\gamma ),$$
such that $\Delta (x)=\sum_{\beta,\gamma\in \ndN_0^n}x_{\beta,\gamma}$.
For any $\beta,\gamma \in \ndN_0^n$, the map
$$\NA (V)\to \NA(V)(\beta)\otimes \NA(V)(\gamma),\quad x\mapsto
x_{\beta,\gamma},$$
is linear. The skew-derivations $d_i$ with $i\in I$ are closely related to these
maps:
\begin{align} \label{eq:Deltapart}
  \Delta_{\alpha_i,\alpha-\alpha_i}(x)=x_i\otimes d_i(x)
\end{align}
for any $i\in I$, $\alpha \in \ndN_0^n$, and any
homogeneous element $x\in \NA (V)$ of degree $\alpha $.

%Other important tools when working with Nichols algebras are the  Dynkin
%diagramm and reflections. They also can be introduced in general.

%Let $q=q_{11}$, $r=q_{12}q_{21}$, $s=q_{22}$. Then the {\bf Dynkin diagram} of $V$ is
%\begin{center}
%	\begin{tikzpicture}[scale=.3]
%		\coordinate[label={[yshift=1pt]$q$}] (A) at (0,0);
%		\coordinate[label={[yshift=3pt]$r$}] (B) at (2,0);
%		\coordinate[label={[yshift=3pt]$s$}] (C) at (4,0);
%	\draw (A)--(C);
%	\fill (A) circle (8pt);
%	\fill (C) circle (8pt);
%\end{tikzpicture}
%\end{center}

Next we discuss reflections. Let $i\in I$. Assume that
for any $j\in I\setminus \{i\}$ there
exists $k\in\ndN_0$ such that $(k+1)_{q_{ii}}(1-q_{ii}^kq_{ij}q_{ji})=0$.
Following \cite{H2006}, we set $c_{ii}=2$ and for any $j\in I\setminus \{i\}$ we
define
$$c_{ij} = -\min\{k\in\ndN_0 \mid (k+1)_{q_{ii}}(1-q_{ii}^kq_{ij}q_{ji})=0\}.$$
Let $s_i\in\mathrm{GL}(\mathbb{Z}^n)$ be given by
$s_i(\alpha_j)=\alpha_j - c_{ij}\alpha_i$ for any $j\in I$. The {\bf reflection
of} $V$ {\bf on the} $i${\bf -th vertex} is the braided vector space $R_i(V)$ with
basis $y_1,\dots,y_n$ such that for all $j,k\in I$,
$$ c(y_j\otimes y_k) = q'_{jk} y_k\otimes y_j,$$
where $q'_{jk} = \chi(s_i(\alpha_j),s_i(\alpha_k)) =
q_{jk}q_{ik}^{-c_{ij}}q_{ji}^{-c_{ik}}q_{ii}^{c_{ij}c_{ik}}$.

In the remaining part of the paper we restrict our attention to the case $n=2$.
Let $q=q_{11}$, $r=q_{12}q_{21}$, and $s=q_{22}$.
We introduce special elements of $\NA (V)$ and spell out some easy properties of
them.

For all $k\in \ndN_0$ we define inductively $u_k\in T(V)$ by
$$ u_0=x_2,\quad u_k=x_1u_{k-1}-q^{k-1}q_{12}u_{k-1}x_1$$
for $k\ge 1$. Note that then $u_k=[x_1^kx_2]$ for any $k\in \ndN_0$.
Let $\ad $ denote the adjoint action of $T(V)$ on itself.
Then
$$ \ad x_1(y)=x_1y-\chi(\alpha_1,\deg(y))yx_1 $$
for any homogeneous element $y\in T(V)$.
In particular,
\begin{align} \label{eq:adx1u}
 \ad x_1(u_k)=u_{k+1}
\end{align}
for any $k\in \ndN_0$.

\begin{lemm} \label{le:d1ad}
  Let $m\in \ndN_0$. Then
\begin{align*}
  d_1\big( (\ad x_1)^m(y)\big)
  =&\,q^m(\ad x_1)^m(d_1(y))\\
  &\,+(m)_q(1-q^{m-1}\chi(\alpha_1,\alpha)\chi(\alpha,\alpha_1))
  (\ad x_1)^{m-1}(y)
\end{align*}
for any homogeneous element $y\in T(V)$ of degree $\alpha \in \ndN_0^n$.
\end{lemm}

\begin{proof}
Let $y\in T(V)$ be a homogeneous element of degree $\alpha \in \ndN_0^n$.
Then
\begin{align*}
 d_1(\ad x_1(y))=&\,d_1(x_1y-\chi(\alpha_1,\alpha)yx_1)\\
=&\,y+qx_1d_1(y)-\chi(\alpha_1,\alpha)(d_1(y)x_1+\chi(\alpha,\alpha_1)y)\\
=&\,(1-\chi(\alpha_1,\alpha)\chi(\alpha,\alpha_1))y+q\ad x_1(d_1(y)).
\end{align*}

Now we prove the lemma by induction on $m$. For $m=0$ the claim is trivial.
Let now $m\in \ndN $ and assume that the claim holds for $m-1$.
Let $y\in T(V)$ be a homogeneous element of degree $\alpha \in \ndN_0$
and let $\beta=\alpha+(m-1)\alpha_1$, $q_\alpha=\chi(\alpha_1,\alpha)\chi(\alpha,\alpha_1)$.
Then, by using the above formula, we conclude that
\begin{align*}
 d_1\big((\ad x_1)^m(y)\big)=&\,
 q\ad x_1\big(d_1\big( (\ad x_1)^{m-1}(y)\big)\big)\\
 &\,+(1-\chi(\alpha_1,\beta )\chi(\beta,\alpha_1))(\ad x_1)^{m-1}(y)\\
 =&\,q\ad x_1\big( q^{m-1}(\ad x_1)^{m-1}(d_1(y))\big)\\
 &\,+q\ad x_1\big( (m-1)_q(1-q^{m-2}q_\alpha )(\ad x_1)^{m-2}(y)\big)\\
 &\,+(1-q^{2m-2}q_\alpha)(\ad x_1)^{m-1}(y)
\end{align*}
because of the induction hypothesis. {}From this one obtains the claim
for $m$.
\end{proof}

\begin{rema} \label{re:Deltau}
(1) It is well-known that
\begin{align}
 \label{eq:Deltau}
	\Delta(u_k)=u_k\otimes 1+\sum_{i=0}^k {k \choose i}_{\!\!q}\prod_{j=k-i}^{k-1}(1-q^jr)
  x_1^i\otimes u_{k-i}
\end{align}
for any $k\in \ndN_0$.
Hence from Equation~\eqref{eq:Deltapart} we conclude that
\begin{align}\label{eq:di}
  d_1(u_k)=(k)_q(1-q^{k-1}r)u_{k-1},\quad d_2(u_k)=\delta_{k0}1
\end{align}
for any $k\in \ndN_0$.

(2) For all $m\in \ndN_0$ let
$$b_m=\prod_{j=0}^{m-1}(1-q^jr).$$
In particular, $b_0=1$. Then Remark~\ref{re:kerd} implies that
$u_k=0$ in $\NA (V)$ if and only if $(k)_q^!b_k=0$.
\end{rema}

Later on, often a normalization of $u_k$, $k\in \ndN_0$, will be very useful.
For all $k\in \ndN _0$ let
\begin{align*}
  \dpu_k=\begin{cases} \frac{1}{(k)_q^!b_k}u_k & \text{if $(k)_q^!b_k\ne 0$,}\\
	0 & \text{otherwise.}
\end{cases}
\end{align*}
The following equations for $k\in \ndN_0$ with $(k)_q^!b_k\ne 0$
follow directly from the analogous formulas for $u_k$.
\begin{align} \label{eq:Deltadpu}
	\Delta(\dpu_k)=&\,\dpu_k\otimes 1+\sum_{i=0}^k\frac{x_1^i}{(i)_q^!}\otimes
	\dpu_{k-i},\\
	\label{eq:ddpu}
	d_1(\dpu_k)=&\,\dpu_{k-1},\quad d_2(\dpu_k)=\delta_{k0}1,\\
	\label{eq:addpu}
  \ad x_1(\dpu_k)=&\,(k+1)_q(1-q^kr)\dpu_{k+1},
\end{align}
where $\dpu_{-1}=0$ in \eqref{eq:ddpu}.

We end the section with a lemma.

\begin{lemm} \label{le:dZ}
	Let $k\in \ndN_0$ such that $(k)_q^!b_k\ne 0$ and let
	$\lambda_0,\dots,\lambda_k\in \fie $. Let
	$Z=\sum_{i=0}^k\lambda_i(-q_{21})^i\dpu_i\dpu_{k-i}$ in $T(V)$. Then
	\begin{align*}
	  d_1(Z)=&-q_{21}\sum_{i=0}^{k-1}(\lambda_{i+1}-q^i\lambda_i)(-q_{21})^i
	  \dpu_i\dpu_{k-1-i},\\
		d_2(Z)=&(\lambda_0+(-r)^ks\lambda_k)\dpu_k.
	\end{align*}
\end{lemm}

\begin{proof}
	This follows directly from Equations~\eqref{eq:dix} and \eqref{eq:ddpu}.
\end{proof}

\section{Multiplicities}
\label{se:multi}

We use the notation from the previous section.
In this section we determine for all $m\in \ndN_0$
the set of root vectors of $\NA (V)$ of degree
$m\alpha_1+2\alpha_2$.

\begin{lemm} \label{le:u_m^2}
	Let $m\in \ndN $ such that $u_m\ne 0$ in $\NA (V)$.
	Then $u_m^2=0$ in $\NA (V)$ if and only if $q^{m^2}r^ms=-1$ and
	$u_{m+1}=0$ in $\NA (V)$.
\end{lemm}

\begin{proof}
	Assume first that $u_m^2=0$. For all $k\in \ndN_0$ let
	$\beta_k=k\alpha_1+\alpha_2$. Using Equation~\ref{eq:Deltau} we obtain that
	$$\Delta_{\beta_m,\beta_m}(u_m^2)=(1+q^{m^2}r^ms)u_m\otimes u_m.$$
	Since $u_m\ne 0$, we conclude that $q^{m^2}r^ms=-1$. Moreover,
	\begin{align*}
	  \Delta_{\beta_{m+1},\beta_{m-1}}&(u_m^2)\\
		=&\,
  	(m)_q(1-q^{m-1}r)(u_mx_1+q^{m(m-1)}r^{m-1}sq_{21}x_1u_m)\otimes u_{m-1}\\
		=&\,(m)_q(1-q^{m-1}r)q^{m(m-1)}r^{m-1}sq_{21}u_{m+1}\otimes u_{m-1}.
  \end{align*}
	Again, since $u_m\ne 0$, from Remark~\ref{re:Deltau}(2) it follows that
	$u_{m+1}=0$.

	Conversely, assume that $q^{m^2}r^ms=-1$ and that $u_{m+1}=0$.
	Then $x_1u_m=q^mq_{12}u_mx_1$, and hence
	\begin{align*}
		\Delta_{2\beta_m-\alpha_1,\alpha_1}(u_m^2)=&\,0,\\
		\Delta_{\beta_{2m},\alpha_2}(u_m^2)=&\,
		b_m\big( (x_1^m\ot u_0)(u_m\ot 1)+u_mx_1^m\ot u_0\big)\\
		=&\,b_m(q_{21}^msq^{m^2}q_{12}^m+1)u_mx_1^m\ot u_0\\
		=&\,0.
	\end{align*}
	Since $\NA (V)$ is a strictly graded coalgebra, it follows that $u_m^2=0$.
\end{proof}

%\begin{lemm} Let $m\in \ndN$. Assume that $u_m\ne 0$, $u_{m+1}=0$,
%and $q^{m^2}r^ms=-1$. Then $[1^m21^k2]$ is not a root vector for any $1\le k\le m$.
%\end{lemm}

%\begin{proof}
%The assumptions imply that $[1^m2]^2=0$. Thus the claim holds
%for $k=m$. We calculate
%$\Delta([1^m2]^2)$.
%Equation~\eqref{eq:Deltau} implies that
%$\Delta(u_m^2)$ is a linear combination of terms of the form
%$x_1^{2m-i-j}\otimes u_iu_j$ with $0\le i,j\le m$, up to terms in $(x_2)\otimes
%\NA (V)$. Let $k\in \{1,\dots ,m-1\}$.
%The coefficient of $x_1^{m-k}\otimes u_mu_k$ in the above sum is
%$${m\choose k}_{\!\!q}\prod_{j=k}^{m-1}(1-q^jr)(1\otimes u_m)(x_1^{m-k}\otimes
%u_k),$$
%which is non-zero since $u_m\ne 0$. We conclude that there is a relation of the
%form
%$$ \sum_{j=0}^{m-k}a_j u_{m-j}u_{k+j}=0$$
%in $\NA (V)$ such that $a_0=1$. Therefore $[1^m21^k2]$ is not a root vector.
%\end{proof}

\begin{lemm}\label{le:0.7}
	Let $k,l\in \ndN _0$ with $k>l$. Assume that $(k)_q^!b_k\ne 0$
        and that
	$[x_1^kx_2x_1^{l+1}x_2]$ is not a root vector. Then $[x_1^kx_2x_1^lx_2]$
	is not a root vector.
\end{lemm}

\begin{proof}
	Lyndon words of degree $(k+l+1)\alpha_1+2\alpha_2$, which are larger than
  $x_1^kx_2x_1^{l+1}x_2$, are of the form $x_1^mx_2x_1^{k+l+1-m}x_2$ with
	$(k+l+1)/2<m<k$. Hence
	the assumption implies that there exists $(\lambda_i)_{l+1\le i\le k}\in
	\fie^{k-l}$ such that $\lambda_k=1$ and
	$\sum_{i=l+1}^k \lambda_i \dpu_i\dpu_{k+l+1-i}=0$ in $\NA (V)$.
	Then
	$$
	  d_1\Big(\sum_{i=l+1}^k \lambda_i \dpu_i\dpu_{k+l+1-i}\Big)
	  =\sum_{i=l+1}^k
	\lambda_i(\dpu_{i-1}\dpu_{k+l+1-i}+q^iq_{21}\dpu_i\dpu_{k+l-i})=0.$$
	The coefficient of $\dpu_k\dpu_l$ in the last expression is $q^kq_{21}$
	and hence the root vector candidate $[x_1^kx_2x_1^lx_2]$ is not
	a root vector.
\end{proof}

\begin{prop} \label{cor:48}
Let $k,l\in \ndN $ with $k\ge l$.
Assume that $[x_1^kx_2x_1^{l}x_2]$ is a root vector candidate but not a root
vector. Then $[x_1^{k+1}x_2x_1^{l-1}x_2]$ is not a root vector.
\end{prop}

\begin{proof}
First assume that $k=l$. Then $[x_1^{k+1}x_2]=0$ by Lemma~\ref{le:u_m^2}.
Therefore $[x_1^{k+1}x_2x_1^{l-1}x_2]$ is not a root vector.
Suppose now that $k>l$. Then $[x_1^kx_2x_1^{l-1}x_2]$ is not a root vector by
Lemma~\ref{le:0.7}, and Remark~\ref{re:rvcrit} implies that
$[x_1^{k+1}x_2x_1^{l-1}x_2]$ is not a root vector.
\end{proof}

For any $n\in \ndN _0$ let
\begin{align}
 U_n=&\,\bigoplus_{i=0}^n \fie u_iu_{n-i}\subseteq T(V),&
 U_n'=&\,\bigoplus_{i=0}^{n-1} \fie u_iu_{n-i}\subseteq T(V).
\end{align}

The subspaces $U_n$ and $U'_n$ will appear at several places as technical tools.
Here we discuss some elementary results related to them.

\begin{lemm} \label{le:adxinjective}
  The map $\ad x_1:U_m\to U_{m+1}$
  is injective for any $m\in \ndN_0$.
\end{lemm}

\begin{proof}
  Let $m\in \ndN_0$, $\lambda_0,\dots,\lambda_m\in \fie $,
and $v=\sum_{i=0}^m \lambda_iu_iu_{m-i}$. Then
\begin{align*}
  \ad x_1(v)
  =\sum_{i=0}^m \lambda_i(u_{i+1}u_{m-i}+q^iq_{12}u_iu_{m+1-i})
\end{align*}
by Equation~\eqref{eq:adx1u}. Assume that $v\ne 0$.
Let $0\le j\le m$ such that $\lambda_j\ne 0$
and either $j=m$ or $\lambda_{j+1}=0$.
Then the coefficient of $u_{j+1}u_{m-j}$ in the above
expression is $\lambda_j$,
and hence $\ad x_1(v)\ne 0$.
\end{proof}

\begin{prop}\label{pr:linsys}
Let $k\in \ndN _0$ such that $(k)_q^!b_k\ne 0$.
Let $v\in U_k'\cap \ker (\pi )$ and let $\mu_0,\dots,\mu_{k-1}\in \fie $
such that
$$d_1(v)=\sum_{i=0}^{k-1} \mu_i(-q_{21})^i\dpu_i\dpu_{k-1-i}.$$
Then $\sum_{i=0}^{k-1}q^{-i(i+1)/2}\mu_i=0$.
\end{prop}

\begin{proof}
For any $\lambda=(\lambda_0,\dots,\lambda_k)\in \fie^{k+1}$ let
$\bar \mu(\lambda)=(\mu_i(\lambda))_{0\le i<k}\in \fie ^k$ such that
$$\mu_i(\lambda)=\lambda_{i+1}-\lambda_iq^i $$
whenever $0\le i<k$.
Let $W=\{\lambda\in \fie^{k+1}\mid \lambda_0=0\}$. Then the linear map
$\bar \mu:W\to \fie^k$,
$\lambda\mapsto \bar \mu(\lambda)$, is bijective. The inverse map is given
by
\begin{align} \label{eq:barmuinverse}
  \bar \mu^{-1}(\mu_0,\dots,\mu_{k-1})=(\lambda_i)_{0\le i\le k},\quad
	\lambda_i=\sum_{j=0}^{i-1}q^{(i+j)(i-j-1)/2}\mu_j.
\end{align}
Now let $\lambda=(\lambda_0,\lambda_1,\dots,\lambda_{k-1},0)\in \fie^{k+1}$
such that
$$v=\sum_{i=0}^k\lambda_i(-q_{21})^i\dpu_i\dpu_{k-i}.$$
Since $v\in \ker (\pi )$, it follows from Remark~\ref{re:kerd}
that $d_2(v)\in \ker (\pi )$.
Since $v\in U_k'$ and $u_k\ne 0$ in $\NA (V)$, Lemma~\ref{le:dZ} implies that $\lambda_0=0$,
that is, $\lambda \in W$. Then we obtain from Equation~\eqref{eq:barmuinverse}
and from $\lambda_k=0$ that $\sum_{j=0}^{k-1}q^{-j(j+1)/2}\mu_j(\lambda)=0$.
Moreover,
\begin{align*}
d_1(v)=&\,
-q_{21}\sum_{i=0}^{k-1}\mu_i(\lambda)(-q_{21})^i\dpu_{i}\dpu_{k-1-i}
\end{align*}
by Lemma~\ref{le:dZ}, and hence $\mu_j=-q_{21}\mu_j(\lambda)$ for any $0\le j<k$.
This implies the claim.
\end{proof}

The following elements of $T(V)$
will play a fundamental role in Theorem~\ref{main}.

\begin{defi}
For all $k\in \ndN_0$ with $(k)_q^!b_k\ne 0$, let
\begin{align} \label{eq:Pk}
 P_k=\sum_{i=0}^k (-q_{21})^iq^{i(i-1)/2}\dpu_i\dpu_{k-i}\in T(V).
\end{align}
\end{defi}

\begin{lemm} \label{d(Pk)}
Let $k\in \ndN_0$ with $(k)_q^!b_k\ne 0$. Then $P_k=0$ in $\NA (V)$ if and only if
$q^{k(k-1)/2}(-r)^ks=-1$.
\end{lemm}

\begin{proof}
 By Lemma~\ref{le:dZ},
\begin{align*}
 d_1(P_k)=&\,0, \quad
 d_2(P_k)=(1+(-r)^ksq^{k(k-1)/2})\dpu_{k}.
\end{align*}
Since $(k)_q^!b_k\ne 0$, the claim follows from this and from Remark~\ref{re:kerd}.
\end{proof}

We also introduce a family of elements $\su(k,t)$ of $T(V)$, which are related to the
elements $P_k$ by Lemmas~\ref{le:adx1su} and \ref{ad pi} below.
Those lemmas themself are needed for Lemma~\ref{le:411}, which is a crucial
ingredient of the proof of Theorem~\ref{main}.

\begin{defi} \label{de:Skt}
For all $k, t\in \ndN _0$ with $0\le t\le k$ and $(k)_q^!b_k\ne 0$ let
\begin{align*}
	\su(k,t)=& \sum_{i=t}^k(-q_{21})^iq^{(i-t)(i-t-1)/2}{i\choose t}_{\!\!q}
	\dpu_i\dpu_{k-i} \in T(V).
\end{align*}
In particular, $\su(k,0)=P_k$.
\end{defi}

\begin{lemm} \label{le:adx1su}
	Let $k, t\in \ndN _0$ with $0\le t\le k$ such that $(k+1)_q^!b_{k+1}\ne 0$. Then
\begin{align*}
	q_{12}^{-1}\ad x_1(\su(k,t))&=q^t(1-q^{k-t}r)(k+1-t)_q\su(k+1,t)\\
                 &\qquad +r^{-1}(q^{2k-t}r^2-1)(t+1)_q\su(k+1,t+1).
\end{align*}
\end{lemm}

\begin{proof}
	First note that for $0\le i\le k$ we get
	\begin{align*}
		\ad x_1(\dpu_i\dpu_{k-i})=&\,(i+1)_q(1-q^ir)\dpu_{i+1}\dpu_{k-i}\\
		&\,+q^iq_{12}(k+1-i)_q(1-q^{k-i}r)\dpu_i\dpu_{k+1-i}
	\end{align*}
	by Equation~\eqref{eq:addpu}. Moreover, $(k+1)_q^!\ne 0$. Hence
\begin{align*}
	&\ad x_1(\su(k,t))\\
	&=\sum_{i=t}^k (-q_{21})^iq^{(i-t)(i-t-1)/2}(i+1)_q(1-q^ir){i\choose t}_q
	\dpu_{i+1}\dpu_{k-i}\\
	&\quad +\sum_{i=t}^k(-q_{21})^iq^{(i-t)(i-t-1)/2}
	q^iq_{12}(k+1-i)_q(1-q^{k-i}r)
	{i\choose t}_{\!\!q}\dpu_i\dpu_{k+1-i}\\
	&=-q_{12}r^{-1}\sum_{i=t+1}^{k+1}(-q_{21})^iq^{(i-t-1)(i-t-2)/2}
	(i)_q(1-q^{i-1}r){i-1\choose t}_{\!\!q} \dpu_i\dpu_{k+1-i}\\
	&\quad +q_{12}q^t\sum_{i=t}^{k+1}(-q_{21})^iq^{(i-t)(i-t+1)/2}(k+1-i)_q(1-q^{k-i}r)
	{i\choose t}_{\!\!q}\dpu_i\dpu_{k+1-i}.
\end{align*}
Now in the first term we replace $(i)_q{i-1\choose t}_{\!\!q}$ by
$(t+1)_q{i\choose t+1}_{\!\!q}$ and $(i-t)_q{i\choose t}_{\!\!q}$, respectively.
We then rewrite this first term as
\begin{align}
	\label{eq:adx1SA}
  &-q_{12}r^{-1}(t+1)_q\sum_{i=t+1}^{k+1}(-q_{21})^iq^{(i-t-1)(i-t-2)/2}
	{i\choose t+1}_{\!\!q}\dpu_i\dpu_{k+1-i}\\
	\label{eq:adx1SB}
	&\quad +q_{12}q^t\sum_{i=t+1}^{k+1}(-q_{21})^iq^{(i-t-1)(i-t)/2}
  (i-t)_q{i\choose t}_{\!\!q}\dpu_i\dpu_{k+1-i}.
\end{align}
The second term of $\ad x_1(\su(k,t))$ can be written as
\begin{align}
	\label{eq:adx1SC}
	&q_{12}q^t\sum_{i=t}^{k+1}(-q_{21})^iq^{(i-t)(i-t-1)/2}q^{i-t}(k+1-i)_q
  {i\choose t}_{\!\!q}\dpu_i\dpu_{k+1-i}\\
	\label{eq:adx1SD}
	&\quad -q_{12}q^kr\sum_{i=t}^{k+1}(-q_{21})^iq^{(i-t)(i-t-1)/2}(k+1-i)_q
	{i\choose t}_{\!\!q}\dpu_i\dpu_{k+1-i}.
\end{align}
Now \eqref{eq:adx1SA} is equal to $-q_{12}r^{-1}(t+1)_q\su(k+1,t+1)$, and the sum
of \eqref{eq:adx1SB} and \eqref{eq:adx1SC} is equal to $q_{12}q^t
(k+1-t)_q\su(k+1,t)$.
Finally, in \eqref{eq:adx1SD} we replace $(k+1-i)_q$ by
$(k+1-t)_q-q^{k+1-i}(i-t)_q$ and $(i-t)_q{i\choose t}_q$ by
$(t+1)_q{i\choose t+1}_q$. Thus \eqref{eq:adx1SD} is equal to
$$ -q_{12}q^kr(k+1-t)_q\su(k+1,t)
+q_{12}q^{2k-t}r(t+1)_q\su(k+1,t+1).$$
This implies the lemma.
\end{proof}

\begin{lemm} \label{ad pi}
	Let $m,k \in \ndN_0$ such that $(k+m)_q^!b_{k+m}\ne 0$. Then
\begin{align*}
	&q_{12}^{-m}(\ad x_1)^m(P_k)
	=
	\sum_{i=0}^m \frac{(m)_q^!}{(m-i)_q^!}
	\lambda_{(m-i,k)}%\left(\prod_{j=1}^{m-i}(1-q^{k+j-1}r)(k+j)_q\right)
	\beta_{(i,m,k)}\su(k+m,i)%\prod_{j=1}^i (q^{2k+m-j}r-r^{-1})\su(k+m,i),
\end{align*}
where for any $i,n,m'\in \ndN_0$,
\begin{align*}\lambda_{(n,k)}=\prod_{j=1}^{n}(1-q^{k-1+j}r)(k+j)_q,\quad
      \beta_{(i,m',k)}=\prod_{j=1}^{i}(q^{m'+2k-j}r-r^{-1}).
		\end{align*}
\end{lemm}

\begin{proof}
Note first that for any $i,n\in \ndN_0$,
\begin{align} \label{eq:newlambda}
	\lambda_{(n+1,k)}=&\,(1-q^{k+n}r)(k+n+1)_q\lambda_{(n,k)},\\
	\label{eq:newbeta1}
	\beta_{(i,m,k)}=&\,
	(q^{m+2k-i}r-r^{-1})\beta _{(i-1,m,k)},\\
	\label{eq:newbeta2}
	\beta_{(i,m+1,k)}=&\,
	(q^{m+2k}r-r^{-1})\beta _{(i-1,m,k)}.
\end{align}

We prove the Lemma by induction on $m$.

For $m=0$, both sides of the equation in the lemma are equal to $P_k$.
Assume now that the formula in the Lemma holds for $m$ and that $(k+m+1)_q^!b_{k+m+1}\ne 0$. Then
\begin{align*}
   &q_{12}^{-m-1}(\ad x_1)^{m+1}(P_k)\\
    &=q_{12}^{-1}\ad x_1\left(q_{12}^{-m}(\ad x_1)^{m}(P_k)\right)\\
		&=q_{12}^{-1}\ad x_1 \left(\sum_{i=0}^m
		\frac{(m)_q^!}{(m-i)_q^!}
	\lambda_{(m-i,k)}
	\beta_{(i,m,k)}\su(k+m,i)\right)\\%\prod_{j=1}^i (q^{2k+m-j}r-r^{-1})\su(k+m,i)\right)\\
\end{align*}
by induction hypothesis. Now apply Lemma~\ref{le:adx1su} to obtain that
\begin{align*}
   &q_{12}^{-m-1}(\ad x_1)^{m+1}(P_k)\\
    &=\sum_{i=0}^m
		\frac{(m)_q^!}{(m-i)_q^!}
	\lambda_{(m-i,k)}
	\beta_{(i,m,k)}\cdot \\
	&\qquad q^i(1-q^{k+m-i}r)(k+m+1-i)_q\su(k+m+1,i)\\
	%\prod_{j=1}^i (q^{2k+m-j}r-r^{-1})\su(k+m+1,i)\\
    &\quad +\sum_{i=0}^m
		\frac{(m)_q^!}{(m-i)_q^!}
	\lambda_{(m-i,k)}%\left(\prod_{j=1}^{m-i}(1-q^{k+j-1}r)(k+j)_q\right)\cdot\\
	\beta_{(i,m,k)}\cdot \\%\prod_{j=1}^i (q^{2k+m-j}r-r^{-1})
    &\qquad (q^{2m+2k-i}r-r^{-1})(i+1)_q\su(k+m+1,i+1).
	\end{align*}
	In the first term we use \eqref{eq:newlambda}, in the second we change the
	summation index. Then
	\begin{align*}
   &q_{12}^{-m-1}(\ad x_1)^{m+1}(P_k)\\
    &=\sum_{i=0}^m
		\frac{(m)_q^!}{(m-i)_q^!}
	\lambda_{(m+1-i,k)}
	\beta_{(i,m,k)}
	q^i\su(k+m+1,i)\\
	%\prod_{j=1}^i (q^{2k+m-j}r-r^{-1})\su(k+m+1,i)\\
	&\quad +\sum_{i=1}^{m+1}
		\frac{(m)_q^!}{(m+1-i)_q^!}
	\lambda_{(m+1-i,k)}%\left(\prod_{j=1}^{m-i}(1-q^{k+j-1}r)(k+j)_q\right)\cdot\\
	\beta_{(i-1,m,k)}\cdot \\%\prod_{j=1}^i (q^{2k+m-j}r-r^{-1})
    &\qquad (q^{2m+2k+1-i}r-r^{-1})(i)_q\su(k+m+1,i).
	\end{align*}
  Thus it remains to show that
	\begin{align*}
	  (m+1-i)_qq^i\beta_{(i,m,k)}+(q^{2m+2k+1-i}r-r^{-1})(i)_q\beta_{(i-1,m,k)}\qquad\\
	  =(m+1)_q\beta_{(i,m+1,k)}
	\end{align*}
	for any $0\le i\le m+1$. The latter is easily done by expressing
	$\beta_{(i,m,k)}$ and $\beta_{(i,m+1,k)}$ via $\beta_{(i-1,m,k)}$
	using Equations~\eqref{eq:newbeta1} and \eqref{eq:newbeta2}, respectively, and
	then comparing coefficients.
This proves the claim for $m+1$.
\end{proof}

Recall the definitions of $Q_1^{k,m},Q_2^{k,m}\in \ndZ[q,r]$ from Lemma~\ref{le:Q1Q2}. In
this section we view $Q_1^{k,m},Q_2^{k,m}$ as elements in $\fie =\fie \otimes
_{\ndZ[q,r]}\ndZ[q,r]$ by identifying $q$ and $r$ in $\ndZ[q,r]$
with $q$ and $r$ in $\fie $, respectively.

\begin{lemm} \label{le:411}
Let $k,m\in \ndN_0$. Suppose that $(k+m+1)_q^!b_{k+m+1}\ne 0$ and that there exists $v\in
U_{k+m+1}'\cap \ker(\pi)$ such that $d_1(v)=(\ad x_1)^m(P_k)$ in $T(V)$.
Then $Q_2^{k,m}=0$.
\end{lemm}

\begin{proof}
	Let $v\in U_{k+m+1}'\cap\ker(\pi)$. Let $\mu_0,\dots,\mu_{k-1}\in \fie $ such
	that
	$$d_1(v)=\sum_{j=0}^{k+m}\mu_j(-q_{21})^j\dpu_j\dpu_{k+m-j}.$$
	Then $\sum_{j=0}^{k+m}q^{-j(j+1)/2}\mu_j=0$ by Proposition~\ref{pr:linsys}.
	Assume now also that $d_1(v)=(\ad x_1)^m(P_k)$.
  Then from Lemma~\ref{ad pi} and Definition~\ref{de:Skt} we obtain that
\begin{align*}
	&q_{12}^m\sum_{i=0}^{m} \frac{(m)_q^!}{(m-i)_q^!}
\lambda_{(m-i,k)} \beta_{(i,m,k)}
\sum_{j=i}^{k+m}q^{(j-i)(j-i-1)/2}q^{-j(j+1)/2}{j\choose i}_{\!\!q}=0
\end{align*}
in $\NA (V)$. (We use the notation in Lemma~\ref{ad pi}.)
Then by Lemma~\ref{le:binomsum} it follows that
\begin{align*}
&\sum_{i=0}^{m} \frac{(m)_q^!}{(m-i)_q^!}
\lambda_{(m-i,k)} \beta_{(i,m,k)}
q^{-(i+1)(2k+2m-i)/2}{k+m+1\choose i+1}_{\!\!q}=0
\end{align*}
Since
\begin{align*}
  \lambda_{(m-i,k)}=&\,
  \prod_{j=1}^{m-i}(1-q^{k-1+j}r) \frac{(k+m-i)_q^!}{(k)_q^!},
  \\
  {k+m+1\choose i+1}_{\!\!q}=&\,
  \frac{(k+m+1)_q^!}{(i+1)_q^!(k+m-i)_q^!},
\end{align*}
the latter implies that
\begin{align*}
&\sum_{i=0}^{m} {m+1\choose i+1}_{\!\!q}
\prod_{j=1}^{m-i}(1-q^{k-1+j}r)\beta_{(i,m,k)}
q^{-(i+1)(2k+2m-i)/2}=0.
\end{align*}
Now substitute $i=m-l$. It follows that
\begin{align*}
&\sum_{l=0}^{m} {m+1\choose l}_{\!\!q}
\prod_{j=0}^{l-1}(1-q^{k+j}r)\prod_{j=1}^{m-l}(q^{2k+m-j}r-r^{-1})
q^{l(2k+l-1)/2}=0.
\end{align*}
The latter is equal to $(-r)^{-m}Q_1^{k,m}$. Thus $Q_2^{k,m}=0$
by Lemma~\ref{le:Q1Q2}.
\end{proof}

Now we introduce the set $\ndJ$ which is crucial for Theorem~\ref{main}
below.

\begin{defi}
  Let $\ndJ =\ndJ_{q,r,s}\subseteq \ndN_0$ be such that $j\in \ndJ$
  if and only if
$$q^{j(j-1)/2}(-r)^js=-1 \text{ and $q^{j+n-1}r^2\ne 1$ for any $n\in \ndJ$, $n<j$.} $$
\end{defi}

\begin{lemm} \label{le:Jsucc}
For any $j\in \ndJ$, the integers $j+1$ and $j+2$ are not in $\ndJ$.
In particular, for any $m\in \ndN_0$,
$$\big| \ndJ \cap [0,m]\big| \le \frac{m}{3}+1.$$
\end{lemm}

\begin{proof}
Let $j\in \ndN_0$ and $t\in \ndN$. Assume that $j,j+t\in \ndJ$. Then
$$ q^{j(j-1)/2}(-r)^js=-1,\quad
q^{(j+t)(j+t-1)/2}(-r)^{j+t}s=-1,$$
and $q^{2j+t-1}r^2\ne 1$.
Hence $q^{t(2j+t-1)/2}(-r)^t=1$. This gives a contradiction
both for $t=1$ and for $t=2$.
\end{proof}

\begin{exam} \label{ex:smallJentries}
  By the definition of $\ndJ$ and by Lemma~\ref{le:Jsucc} the following hold.
\begin{enumerate}
\item $0\in \ndJ $ if and only if $s=-1$.
\item $1\in \ndJ $ if and only if $rs=1$ and $s\ne -1$.
\item $2\in \ndJ $ if and only if $qr^2s=-1$, $s\ne -1$, and $rs\ne 1$.
\end{enumerate}
\end{exam}

For the proof of the next theorem we will need a technicality.

\begin{lemm} \label{le:main}
	Assume that $\mathrm{char}(\fie)=0$.
	Let $k,m\in \ndN _0$, and assume that $b_{k+m+1}\ne 0$ and
  $q^{2k+m}r^2=1$.
	Then $Q_2^{k,m}\ne 0$ in $\fie $.
\end{lemm}

\begin{proof}
	Assume first that $m$ is odd and that $q^{2k+m}r^2=1$. Then
	\begin{align*}
		Q_2^{k,m}=\sum_{i=0}^{(m-1)/2}(q^{2k+m}r^2)^i
		\prod_{i=0}^m(1-q^{k+i}r)=\frac{m+1}2\prod_{i=0}^m(1-q^{k+i}r).
	\end{align*}
	Since $b_{k+m+1}\ne 0$ and $\mathrm{char}(\fie)=0$, we conclude that
	$Q_2^{k,m}\ne 0$ in $\fie $.

	Assume now that $q^{2k+m}r^2=1$ and that $m$ is even. Let $n=m/2$. Since
	$b_{k+m+1}\ne 0$, it follows that $q^{k+n}r=-1$. Hence
	\begin{align*}
	  Q_2^{k,m}=\sum_{i=0}^{m}(-q^{k+n}r)^i
		\prod_{i=0}^{n-1}(1-q^{k+i}r)\prod_{i=n+1}^m(1-q^{k+i}r).
	\end{align*}
	Thus we again obtain that $Q_2^{k,m}\ne 0$ in $\fie $.
\end{proof}

\begin{theo} \label{main}
	Assume that $\mathrm{char}(\fie )=0$. Let $m\in \ndN_0$ such that
$(m)_q^!b_m\ne 0$. Then the elements $(\ad x_1)^{m-j}(P_j)$ with $j\in \ndJ \cap[0,m]$
form a basis of $\ker(\pi)\cap U_m$.
\end{theo}

\begin{proof}
First note that $P_j\in \ker (\pi)\cap U_j$
for any $j\in \ndJ$ because of Lemma~\ref{d(Pk)}.
Hence $(\ad x_1)^{m-j}(P_j)\in \ker (\pi)\cap U_m$ for any $j\in \ndJ \cap[0,m]$.

Now we prove by induction on $m$ that the elements
$(\ad x_1)^{m-j}(P_j)$ with $j\in \ndJ \cap [0,m]$ are linearly independent.
This is clear for $m=0$. Assume now that $m>0$, and for any $j\in \ndJ\cap[0,m]$ let
$\lambda_j\in \fie $ such that
$$\sum_{j\in \ndJ\cap[0,m]}\lambda_j(\ad x_1)^{m-j}(P_j)=0.$$
If $m\notin \ndJ$, then
$\sum_{j\in \ndJ\cap[0,m]}\lambda_j(\ad x_1)^{m-1-j}(P_j)=0$
by Lemma~\ref{le:adxinjective}. Hence $\lambda_j=0$ for all $j\in \ndJ\cap[0,m]$
by induction hypothesis.

Assume now that $m\in \ndJ$. Lemma~\ref{le:dZ} implies that $d_1(P_n)=0$ for any
$n\in \ndN_0$, and hence
$$\sum_{j\in \ndJ\cap[0,m-1]}\lambda_jd_1\big((\ad x_1)^{m-j}(P_j)\big)=0.$$
{}From Lemma~\ref{le:d1ad} then it follows that
$$\sum_{j\in \ndJ\cap[0,m-1]}\lambda_j(m-j)_q(1-q^{m-j-1}q^{2j}r^2)
(\ad x_1)^{m-1-j}(P_j)=0.$$
Note that $q^{m+j-1}r^2\ne 1$ for all $j\in \ndJ\cap [0,m-1]$ because of $m\in \ndJ$.
Moreover, $(m)_q^!\ne 0$ by assumption. Therefore
induction hypothesis implies that $\lambda_j=0$
for all $j\in \ndJ\cap [0,m-1]$. Then clearly $\lambda_m=0$ holds, too.

It remains to show that
\begin{align} \label{eq:dimkerpi}
	\dim \big(\ker (\pi)\cap U_m\big) =\big|\ndJ\cap [0,m]\big|.
\end{align}
Again we proceed by induction on $m$.
Note that $P_0=u_0^2\in \ker (\pi )$ if and only if $s=-1$, that is,
$0\in \ndJ $,
according to Lemma~\ref{d(Pk)}. Thus the claim holds for $m=0$.

Let now $m\in \ndN $.
Induction hypothesis and the first part of the proof of the Theorem imply that
the elements $(\ad x_1)^{m-1-j}(P_j)$, where $j\in \ndJ \cap [0,m-1]$, form a
basis of $\ker (\pi)\cap U_{m-1}$.
Since $\ad x_1$ is injective by Lemma~\ref{le:adxinjective} and since
$\ad x_1(\ker (\pi))\subseteq \ker (\pi)$, we further obtain that
$$\dim \big(\ker(\pi )\cap U_m\big)\ge \dim \big( \ker(\pi)\cap U_{m-1}\big).$$

Assume first that $\dim \big(\ker(\pi )\cap U_m\big)=\dim \big( \ker(\pi)\cap U_{m-1}\big)$.
Then the elements
$(\ad x_1)^{m-j}(P_j)$, where $j\in \ndJ \cap [0,m-1]$, form a basis of $\ker (\pi )\cap U_m$.
Moreover, the linear independence of the elements $(\ad x_1)^{m-j}(P_j)$, $j\in \ndJ \cap
[0,m]$, implies that $m\notin \ndJ$. This proves \eqref{eq:dimkerpi}.

Assume now that $\dim \big(\ker(\pi )\cap U_m\big)>\dim \big( \ker(\pi)\cap U_{m-1}\big)$.
Since
$$d_1(\ker(\pi )\cap U_m)\subseteq \ker(\pi)\cap U_{m-1},$$
we conclude that $\ker(\pi)\cap U_m\cap \ker (d_1)\ne 0$.
Since $(m)_q^!b_m\ne 0$,
Lemma~\ref{le:dZ}
implies that $\ker (d_1|U_m)=\fie P_m$.
Hence $P_m\in \ker(\pi)\cap U_m$,
\begin{align} \label{eq:dimkerpi1}
 \dim \big(\ker(\pi )\cap U_m\big)=1+\dim \big( \ker(\pi)\cap U_{m-1}\big),
\end{align}
and for any $j\in \ndJ$ there exists $v_j\in \ker (\pi)\cap U_m$ such that
$$d_1(v_j)=(\ad x_1)^{m-1-j}(P_j).$$
Since $P_m\in \ker(\pi)\cap U_m$,
we obtain from Lemma~\ref{d(Pk)} that
$$q^{m(m-1)/2}(-r)^ms=-1.$$
Further, we
may assume that $v_j\in \ker(\pi)\cap U'_m$ for any $j\in \ndJ\cap[0,m-1]$.
Hence $Q_2^{j,m-1-j}=0$ for any $j\in \ndJ\cap[0,m-1]$ by Lemma~\ref{le:411}.
Since $\mathrm{char}(\fie)=0$, from Lemma~\ref{le:main} we conclude that
$q^{m+j-1}r^2\ne 1$ for any $j\in \ndJ\cap[0,m-1]$. Thus $m\in \ndJ $.
Then Equation~\eqref{eq:dimkerpi} follows from \eqref{eq:dimkerpi1} and from
induction hypothesis.
\end{proof}

\begin{coro} \label{co:rootvectors}
	Assume that $\mathrm{char}(\fie )=0$. Let $k,l\in \ndN_0$
  with $k\ge l$. Suppose that $(k+l)_q^!b_{k+l}\ne 0$, and that
  $q^{k^2}r^ks=-1$ if $k=l$.
  Then the following are equivalent.
	\begin{enumerate}
		\item $[x_1^kx_2x_1^lx_2]$ is a root vector,
    \item $\big|\ndJ\cap [0,k+l]\big|\le l$.
	\end{enumerate}
\end{coro}

\begin{proof}
	By assumption, $[x_1^kx_2x_1^lx_2]$ is a root vector candidate.
	Proposition~\ref{cor:48} and Example~\ref{ex:Lyndon2}
	imply that $[x_1^kx_2x_1^lx_2]$ is a root vector if and only if
	any root vector candidate of degree $(k+l)\alpha_1+2\alpha_2$,
	which is not a root vector, is of the form
	$[x_1^{k_1}x_2x_1^{k_2}x_2]$ with $k_1+k_2=k+l$, $0\le k_2<l$.
	This just means that $\dim \big(\ker (\pi)\cap U_{k+l}\big)\le l$.
	According to Theorem~\ref{main}, the latter is equivalent to
	$\big|\ndJ\cap[0,k+l]\big|\le l$.
\end{proof}

\begin{coro} \label{size}
  Assume that $\mathrm{char}(\fie )=0$. Let $m\in \ndN_0$
	such that $(m)_q^!b_m\ne 0$.
Then the multiplicity of $m\alpha_1+2\alpha_2$ is
$$ m'-\big|\ndJ\cap[0,m]\big|,$$
where
\begin{align*}
  m'=\begin{cases} (m+1)/2 & \text{if $m$ is odd,}\\
  m/2 & \text{if $m$ is even and $q^{m^2/4}r^{m/2}s\ne -1$,}\\
  m/2+1 & \text{if $m$ is even and $q^{m^2/4}r^{m/2}s=-1$.}
\end{cases}
\end{align*}
\end{coro}

\begin{proof}
	By Example~\ref{ex:Lyndon2}, $m'$ is just the number of root vector candidates
	of degree $m\alpha_1+2\alpha_2$. Corollary~\ref{co:rootvectors} implies that
	$\big|\ndJ\cap[0,m]\big|$ is the number of root vector candidates of degree
	$m\alpha_1+2\alpha_2$ which are	not root vectors. This implies the claim.
\end{proof}

The following proposition treats the question in Corollary~\ref{size} if the
assumption on $m$ is not satisfied. Recall that $R_1(V)$ is the reflection of
$V$ on the first vertex.

\begin{prop} \label{pr:refl}
  Assume that $\mathrm{char}(\fie )=0$. Let $k,m\in \ndN_0$
	such that $(k)_q^!b_k\ne 0$, $(k+1)_q(1-q^kr)=0$, and $m\ge k$.
	Then the multiplicity of $m\alpha_1+2\alpha_2$ is
	the same as the multiplicity of $(2k-m)\alpha_1+2\alpha_2$ of $\NA (R_1(V))$.
\end{prop}

\begin{proof}
	The claim is a very special case of the invariance of multiplicities under
	reflections, which was proved in \cite{H2006}.
\end{proof}

\begin{rema}
	According to the explanations in Section~\ref{se:prelims},
	in Proposition~\ref{pr:refl} we have $c_{12}=-k$. Hence
	the braiding matrix $(q'_{ij})_{i,j\in\{1,2\}}$ of $R_1(V)$ satisfies
	$$ q'_{11}=q,\quad q'_{12}q'_{21}=r, \quad q'_{22}=s $$
	whenever $q^kr=1$, and
	$$ q'_{11}=q,\quad q'_{12}q'_{21}=q^2r^{-1}, \quad q'_{22}=qr^ks $$
	whenever $q^kr\ne 1$ (and then $(k+1)_q=0$).
	Since $2k-m\le k$, the multiplicity of $(2k-m)\alpha_1+2\alpha_2$ of $\NA
	(R_1(V))$ can be obtained using Corollary~\ref{size} with the set $\ndJ$ for
	$(q'_{ij})_{i,j\in \{1,2\}}$.
\end{rema}

Finally, we discuss the multiplicity of roots in some special cases.

\begin{coro} \label{co:non-roots}
Assume that $\mathrm{char}(\fie )=0$. Let $m\in \ndN_0$.
\begin{enumerate}
\item Assume that $m\in \{1,2,3,4,6\}$ and that $(m)_q^!b_m\ne 0$.
Then $m\alpha_1+2\alpha_2$ is not a root if and only if
$q,r,s$ satisfy the conditions given in Table~\ref{ta:noroot}.
\item Assume that $m=2k+1\ge 5$ is odd and that $(k+3)_q^!b_{k+3}\ne0$.
Then $m\alpha_1+2\alpha_2$ is a root of $\NA (V)$.
\item Assume that $m=2k\ge 8$ and that $(k+4)_q^!b_{k+4}\ne 0$.
Then $m\alpha_1+2\alpha_2$ is a root of $\NA (V)$.
\end{enumerate}
\end{coro}

\begin{proof}
(1) We apply Corollary~\ref{size} case by case.

Assume that $m=1$. Then $m'=1$. Hence
$\alpha_1+2\alpha_2$ is not a root if and only if $\big|\ndJ\cap [0,1]\big|=1$.
According to Example~\ref{ex:smallJentries}, this is equivalent to
$(1+s)(1-rs)=0$.

Assume that $m=2$.
Then $\big|\ndJ\cap [0,2]\big|\le 1$ by Example~\ref{ex:smallJentries},
and equality holds if and only if $(1+s)(1-rs)(1+qr^2s)=0$.
Hence, if $qrs\ne -1$, then $m'=1$ and the claim is proven.
On the other hand,
if $qrs=-1$, then $m'=2$ and hence $2\alpha_1+2\alpha_2$ is a root.
Note that in this case $(1+s)(1-rs)(1+qr^2s)\ne 0$
since
$$(m)_q^!b_m=(2)_q(1-r)(1-qr)\ne 0.$$
Thus the claim is valid also in this case.

Assume that $m=3$. Then $m'=2$. Hence
$3\alpha_1+2\alpha_2$ is not a root if and only if $\big|\ndJ\cap [0,3]\big|=2$.
Due to Lemma~\ref{le:Jsucc}, the latter is only possible if $\ndJ\cap[0,3]=\{0,3\}$.
This means that $s=-1$, $q^3r^3s=1$, and $q^2r^2\ne 1$. Because of $(3)_q^!b_3\ne 0$
we can rewrite this condition to $s=-1$, $(3)_{-qr}=0$.

The conditions for $m=4$ and $m=6$ can be obtained similarly.

(2)
Assume first that $(m)_q^!b_m\ne 0$. By Corollary~\ref{size},
the multiplicity of $m\alpha_1+2\alpha_2$ is $k+1-\big|\ndJ\cap[0,m]\big|$.
By Lemma~\ref{le:Jsucc}, $\big|\ndJ\cap[0,m]\big|\le m/3+1$.
Since $3k-m=k-1>0$, we conclude that $m\alpha_1+2\alpha_2$ is a root of $\NA (V)$.

Assume now that $(m)_q^!b_m=0$.
Since $(k+3)_q^!b_{k+3}\ne 0$ by assumption,
for the Cartan matrix entry
$c_{12}$ we obtain that $k+3\le -c_{12}<m$. Moreover,
$$s_1(m\alpha_1+2\alpha_2)=(-2c_{12}-m)\alpha_1+2\alpha_2 $$
and $-2c_{12}-m$ is odd and lesser than $-c_{12}$. Moreover,
$$-2c_{12}-m-5=-2c_{12}-2k-6\ge 0$$
and hence Proposition~\ref{pr:refl} and the previous paragraph
for $R_1(V)$ imply that $m\alpha_1+2\alpha_2$ is a root of $\NA (V)$.

(3) Similar to the proof of (2). Note that $2k\alpha_1+2\alpha_2$
is always a root if $q^{k^2}r^ks=-1$ and $(k+1)_q^!b_{k+1}\ne 0$
because of Lemma~\ref{le:u_m^2}. Hence only the case where
$q^{k^2}r^ks\ne -1$ has to be considered in detail.
\end{proof}

\begin{table}[!hbp]
\label{ta:noroot}
\begin{tabular}{c|l}
$m\alpha_1+2\alpha_2$ & non-root conditions \\
\hline
$\alpha_1+2\alpha_2$ &   $(1+s)(1-rs)=0$ \\
\hline
$2\alpha_1+2\alpha_2$ &$(1+s)(1-rs)(1+qr^2s)=0$  \\
\hline
$3\alpha_1+2\alpha_2$ & $s=-1,(3)_{-qr}=0$\\
 \hline
 $4\alpha_1+2\alpha_2$ & \tabincell{l}{ $s=-1$, $(3)_{-qr}=0$ or \\
                                        $s=-1$, $q^3r^2=-1$ or\\
                                        $rs=1$, $(3)_{-q^2r}=0$}\\

\hline
$6\alpha_1+2\alpha_2$ & $q=1,s=-1$, $(3)_{-r}=0$
\end{tabular}
\caption{Table for Corollary~\ref{co:non-roots}}
\end{table}

%\bibliographystyle{plain}
%\bibliography{M2}

\end{document}